\documentclass[pdftex,a4paper,abstracton]{scrartcl}
\usepackage[utf8]{inputenc}
\usepackage[english]{babel}
\usepackage{amssymb,amsmath,amsthm}
\usepackage[round]{natbib}

\newtheorem{Satz}{Theorem}
\newtheorem*{Satz*}{Theorem A}
\newtheorem{Lemma}{Lemma}
\theoremstyle{definition}\newtheorem{Def}{Definition}
\theoremstyle{definition}\newtheorem{example}{Example}
\theoremstyle{remark}\newtheorem{Remark}{Remark}
\newcommand{\RR}{\mathbb R}
\newcommand{\NN}{\mathbb N}
\newcommand\numberthis{\addtocounter{equation}{1}\tag{\theequation}}
\newcommand{\defref}[1]{Definition \ref{#1}}
\newcommand{\thref}[1]{Theorem \ref{#1}}
\newcommand{\BIGOP}[1]
{
\mathop{\mathchoice%
{\raise-0.22em\hbox{\huge $#1$}}%
{\raise-0.05em\hbox{\Large $#1$}}{\hbox{\large $#1$}}{#1}}}
\newcommand{\bigtimes}{\BIGOP{\times}}
\newcommand{\BIGboxplus}{\mathop{\mathchoice%
{\raise-0.35em\hbox{\huge $\boxplus$}}%
{\raise-0.15em\hbox{\Large $\boxplus$}}{\hbox{\large $\boxplus$}}{\boxplus}}}

\allowdisplaybreaks
\begin{document}


\title{The Sequential Empirical Process of Nonlinear Long-Range Dependent Random Vectors}
\author{Jannis Buchsteiner\thanks{{E-mail: \texttt{jannis.buchsteiner@rub.de}\newline Research supported by Collaborative Research Center SFB 823 {\em
Statistical modeling of nonlinear dynamic processes.}}}\\ \normalsize{\textit{Fakultät für Mathematik, Ruhr-Universität Bochum, Germany.}}}
\date{}
\maketitle

\begin{abstract}
Let $(G(X_j))_{j\geq1}$ be a multivariate subordinated Gaussian process, which exhibits long-range dependence. We study the asymptotic behaviour of the corresponding sequential empirical process under two different types of subordination. The limiting process is either a product of a deterministic function and a Hermite process as in the one-dimensional case or a sum of various processes of this kind.\\[1ex]
\end{abstract}

\section{Introduction}
For a Gaussian process $(X_j)_{j\geq1}$ and a measurable function $G$ the sequential empirical process $(R_N(x,t))$ corresponding to the subordinated process $(G(X_j))_{j\geq1}$ is given by
\begin{equation}\label{seq}
 R_N(x,t):=\sum_{j=1}^{\lfloor Nt\rfloor}\left(1_{\{G(X_j)\leq x\}}-P(G(X_j)\leq x)\right),~~~x\in\RR,t\in[0,1].
\end{equation}
This process plays an important role in nonparametric statictics, for example in change-point analysis. If the underlying Gaussian process exhibits long-range dependence (LRD) weak convergence was shown by \citet{Dehling}. The limiting process can be represented as the product of a deterministic function and a Hermite process $(Z_m(t))_{0\leq t\leq1}$ which is a fractional Brownian motion if $m=1$ and a non-Gaussian process for $m\geq 2$. A first step in generalizing this result to multivariate observations was done by \citet{Marinucci}. He studied the asymptotics of the empirical process $(R_N(x,1))$ based on a two-dimensional LRD process. Roughly speaking one could find two different approaches in the literature to define LRD for a $p$-dimensional stochastic process $(Y_j)_{j\geq1}$. The first one is to take $p$ independent processes
$(Y_j^{(1)})_{j\geq1},\ldots,(Y_j^{(p)})_{j\geq1}$ such that each of them exhibits LRD, i.e.
\begin{equation}\label{EinleitungLRD}
\operatorname{Cov}(Y_1^{(i)},Y_{k+1}^{(i)})=L_i(k)k^{-D_i},~~~0<D_i<1
\end{equation}
and $L_i$ is slowly varying at infinity. For $p=2$ this construction was used by \citet{Marinucci} and \citet[$D_1=D_2$]{Taufer}. A more general setup was used by \citet{Ho} and \citet{Arcones}. They call a $p$-dimensional Gaussian process LRD if both the covariance function of each component and further the cross-covariance  $\operatorname{Cov}(Y_1^{(i)},Y_{k+1}^{(j)})$ are of type \eqref{EinleitungLRD}. \citet{Pipiras} stated a very precise definition of LRD in time and spectral domain and showed under which conditions these are equivalent.

In the present paper we want to establish new non-central limit theorems for the sequential empirical process based on $p$-dimensional LRD data. In the tradition of the initial work of \citet{Dehling} we will focus on subordinated Gaussian processes. More precisely we will consider two different types of subordination. This approach will help us to make the proofs more transparent. In section \ref{section one-dimensional} the $p$-dimensional observations are generated by a one-dimensional Gaussian process. In section \ref{section multi} the process $(G(X_j))$ is $q$-dimensional, where the underlying Gaussian process itself is $p$-dimensional.  For both cases we prove a weak uniform reduction principle and show that $(R_N(x,t))$ converges weakly to a Hermite process as in the one-dimensional case (section \ref{section one-dimensional}) or to a sum of generalized Hermite processes (section \ref{section multi}), respectively.  

\section{One-dimensional subordination}\label{section one-dimensional}
The simplest way to get a $p$-dimensional subordinated Gaussian random vector is the following construction. Let $(X_j)_{j\geq 1}$ be an one-dimensional stationary Gaussian process with $EX_1=0$ and $EX_1^2=1$. Moreover, let this sequence exhibits long-range dependence such that the covariance function $r(k)=EX_1X_{k+1}$ satisfies
\begin{equation}\label{covariance}
r(k)=k^{-D}L(k),
\end{equation}
where $0<D<1$ and $L$ is a function which is slowly varying at infinity. For any measurable functions $G_1,\ldots,G_p:\RR\rightarrow\RR$ we consider the process
$(Y_j)_{j\geq1}$, where $Y_j$ is a random vector given by
\begin{equation*}
 Y_j:=G(X_j):=(G_1(X_j),\ldots.G_p(X_j)).
\end{equation*}
To simplify reading, we first introduce some notations. We denote by $F$ the common distribution function of
$Y_1$ and by $F_k$ the distribution function of $G_k(X_1)$. For an element $(x^{(1)},\ldots,x^{(p)})\in\RR^p$ we write simply $x$ and by $x\leq y$ we mean $x_i\leq y_i$,
for all $1\leq i\leq p$. The sequential empirical process based on these $p$-dimensional observations can be written just as \eqref{seq} as
\begin{equation*}
 R_N(x,t):=\sum_{j=1}^{\lfloor Nt\rfloor}\left(1_{\{Y_j\leq x\}}-F(x)\right).
\end{equation*}
\subsection{Results}
As in the initial work of \citet{Dehling} the asymptic behaviour of $(R_N(x,t))$ is determined by the leading term of the so-called Hermite expansion. These relates to the collection of Hermite polynomials $(H_n)_{n\geq0}$,
\begin{equation*}
 H_n(y):=(-1)^ne^{y^2/2}\frac{d^n}{dy^n}e^{-y^2/2},
\end{equation*}
which forms an orthogonal basis of $L_2(\varphi(y)dy)$, where $\varphi$ is the standard normal density. Since $(1_{\{G(\cdot)\leq x\}}-F(x))$ is square integrable for
any $x\in\RR^p$, we have the following $L_2$-representation
\begin{equation*}
 1_{\{Y_j\leq x\}}-F(x)=\sum_{q=0}^{\infty}\frac{J_q(x)}{q!}H_q(X_j).
\end{equation*}
The Hermite coefficients are given by the inner product, i.e.
\begin{equation}\label{hc}
 J_q(x)=E(1_{\{Y_j\leq x\}}-F(x))H_q(X_j).
\end{equation}
We call the index $m(x)$ of the first nonzero Hermite coefficient the Hermite rank of $(1_{\{G(\cdot)\leq x\}}-F(x))$. Note that $m(x)\geq1$ for any $x\in\RR^p$.
\begin{Satz}\label{nclt}
Let $(X_j)_{j\geq1}$ be a standard one-dimensional Gaussian process satisfying \eqref{covariance} and let $G:\RR\rightarrow\RR^p$ be a measurable function. Furthermore, let $0<D<1/m$, where $m:=\min\{m(x):x\in\RR^p\}$. Then
\begin{equation*}
 \left\{d_N^{-1}R_N(x,t): (x,t)\in[-\infty,\infty]^p\times [0,1]\right\}
\end{equation*}
converges weakly in $D([-\infty,\infty]^p\times [0,1])$, to
\begin{equation*}
 \left\{\frac{J_m(x)}{m!}Z_m(t): (x,t)\in[-\infty,\infty]^p\times [0,1]\right\},
\end{equation*}
where $(Z_m(t))$ is a Hermite process.
\end{Satz}
\begin{Remark}
It was shown by \citet[Corollary 4.1]{Taqqu1975} that 
\begin{equation*}
d_N^2 \approx N^{2-mD}L(N).
\end{equation*}
\end{Remark}
\begin{Remark}
The Hermite process $(Z_m(t))_{0\leq t\leq1}$  is given by a stochastic integral in spectral domain, more precisely
\begin{equation}\label{Spektraldarstellung}
Z_m(t)=c\int_{\RR^m}^{''}\frac{e^{it(x_1+\ldots+x_m)}-1}{i(x_1+\ldots+x_m)}\prod_{j=1}^m|x_j|^{-(1-D)/2}B(dx_1)\cdots B(dx_m),
\end{equation}
where $B$ is a suitable random spectral measure and $c\in\RR$ is a constant which only depends on $m$ and $D$. For details and further representations see \citet{Taqqu1979} and \citet{Taqqu2010}.
\end{Remark}
\begin{Satz}[Reduction principle]\label{rp}
 Let $(X_j)_{j\geq1}$ be a standard one-dimensional Gaussian process satisfying \eqref{covariance} and let $G:\RR\rightarrow\RR^p$ be a measurable function. Furthermore, let $0<D<1/m$, where $m:=\min\{m(x):x\in\RR^p\}$. Then there exist constants $C,\kappa>0$ such that for any $0<\varepsilon\leq1$
\begin{align}
 P\Biggl(&\max_{n\leq N}\sup_{x\in[-\infty,\infty]^p}d_N^{-1}\left|\sum_{j=1}^n\left(1_{\{Y_k\leq x\}}-F(x)-\frac{J_m(x)}{m!}H_m(X_j)\right)\right|
         >\varepsilon\Biggr)\notag\\
\leq &CN^{-\kappa}(1+\varepsilon^{-3}).\notag
\end{align}
The normalization factor is given by
\begin{equation*}
 d_N^2=\operatorname{Var}\left(\sum_{j=1}^N H_m(X_j)\right).
\end{equation*}
\end{Satz}
The statement of Theorem \ref{rp} is strongly associated to the reduction principle of \citet{Dehling} in the
case $p=1$. They were the first who established such a theorem uniformly in $x$ and $t$. \citet{Giraitis} studied the reduction principle for the seqential
empirical process of long-range dependent moving average data. \citet{Jannis} showed that the reduction principle of Dehling and Taqqu is still valid, if the
càdlàg space is equipped with a weighted supremum norm. In all cases the weak reduction principle can be applied to prove weak convergence of the normalized
sequential empirical processs. 

\subsection{The Hermite Rank}
In order to use Theorem \ref{nclt} for studying the asymptotic behaviour of $R_N(x,t)$ it is important to know the Hermite rank $m$ of $\{1_{\{G(\cdot)\leq x\}}:x\in\RR^p\}$. For $p=1$ we usually have $m\leq 2$, see \citet[pp. 1770]{Dehling}. Therefore the question arises, is there a connection between the Hermite ranks $m_j$ of
$\{1_{\{G_j(\cdot)\leq x\}}:x\in\RR\}$ and $m$?
\begin{Lemma}\label{rank}
Let $m_j$ be the Hermite rank of $\{1_{\{G_j(\cdot)\leq x\}}:x\in\RR\}$, $1\leq j\leq p$, and let $m$ be the Hermite rank of $\{1_{\{G(\cdot)\leq x\}}:x\in\RR^{p}\}$. Then 
\begin{equation*}
m\leq\min\{m_j: 1\leq j\leq p\}.
\end{equation*}
\end{Lemma}
\begin{proof}
For simplicity let $m_1=\min\{m_j:1\leq j\leq p\}$. Then there exist $x\in\RR$ s.t. $E(1_{\{G_1(X_1)\leq x\}}H_{m_1}(X_1))\neq 0$. By dominated convergence we have
\begin{align*}
\lim_{n\rightarrow\infty}&E(1_{\{G_1(X_1)\leq x,G_2(X_1)\leq n,\ldots,G_p(X_1)\leq n\}}H_{m_1}(X_1))\\
=&E(\lim_{n\rightarrow\infty}1_{\{G_1(X_1)\leq x,G_2(X_1)\leq n,\ldots,G_p(X_1)\leq n\}}H_{m_1}(X_1))\\
=&E(1_{\{G_1(X_1)\leq x\}}H_{m_1}(X_1))\neq 0.
\end{align*}
Therfore it exists an index $n_0$ so that $E(1_{\{G_1(X_1)\leq x,G_2(X_1)\leq n_0,\ldots,G_p(X_1)\leq n_0\}}H_{m_1}(X_1))\neq 0$ and this implies $m\leq m_1$.
\end{proof}
\begin{example}
Assume we have $\min\{m_j: 1\leq j\leq p\}=1$. Then we get by Lemma \ref{rank} $m=1$. In other words, if there is at least one function $G_j$ s.t.  $\{1_{\{G_j(\cdot)\leq x\}}:x\in\RR\}$ has Hermite rank $1$ then $R_N(x,t)$ converges weakly to a fractional Brownian motion.
\end{example}
\begin{example}
Let $p=2$, $G_1(s)=s^2$ and $G_2(s)=1_A(s)$, where $A=(-\infty,-c)\cup(0,c)$ and $c=(-2\ln(1/2))^{1/2}$. Note that $\{s\in\RR:G_2(s)\leq x\}\in\{\emptyset, A^c, \RR\}$ for all $x\in\RR$ and
\begin{align*}
\int\limits_{-c}^0 x\varphi(x)dx+\int\limits_c^{\infty} x\varphi(x)dx=0,\\
\int\limits_{-c}^0 (x^2-1)\varphi(x)dx+\int\limits_c^{\infty} (x^2-1)\varphi(x)dx=0.
\end{align*}
Using numeric integration we get $E(1_{\{G_2(X_1)\leq x\}}H_{3}(X_1))\neq 0$ for $0\leq x<1$. This implies $\{1_{\{G_2(\cdot)\leq x\}}:x\in\RR\}$ has Hermite rank $3$. Furthermore, we know by \citet[Example 2]{Dehling} that $\{1_{\{G_1(\cdot)\leq x\}}:x\in\RR\}$ has Hermite rank $2$.\\
Now we are in the special situation in which the Hermite rank corresponding to the common observation is really smaller than $2$, since
\begin{equation*}
E(1_{\{G_1(X_1)\leq c^2,G_2(X_1)\leq 0\}}H_{1}(X_1))=\int\limits_{-c}^0 x\varphi(x)dx=2(2\pi)^{-1/2}.
\end{equation*}
 \end{example}
 \section{Multivariate subordination}\label{section multi}
Let $G:\RR^p\rightarrow\RR^q$ a  measurable function and let $(X_j)_{j\geq1}=((X^{(1)}_j,\ldots,X^{(p)}_j))_{j\geq1}$ be a $p$-dimensional stationary Gaussian process with the following properties
\begin{align}
EX^{(i)}_1=&0~~~~1\leq i \leq p\label{a1}\\
EX^{(i)}_1X^{(j}_1=&\delta_{ij}\label{a2}\\
r^{(i,j)}(k):=EX^{(i)}_1X^{(j)}_{k+1}=&c_{ij}L(k)k^{-D},\label{a3}
\end{align}
where $c_{ij}\in\RR$ are constants depending only on $i$ and $j$, $L(k)$ is slowly varying at infinity and $0<D<1$. From now we assume that the $\RR^q$-valued stochastic process $(Y_j)_{j\geq1}$ is given by $Y_j=G(X_j)$.\\ 
The conditions \eqref{a1}, \eqref{a2} and \eqref{a3} were proposed by \citet{Arcones} and they are also compatible with the very general definition of multivariate long-range dependence given by \citet{Pipiras}. 

In order to prove a weak uniform reduction principle we have to  establish a multivariate Hermite decomposition of $R_N(x,t)$. Let $L^2$ be the space of square integrable functions with respect to the $p$-dimensional standard normal distribution. An orthogonal basis of these space is given by the collection of multivariate Hermite polynomials 
\begin{equation*}
H_{l_1,\ldots,l_p}(x)=H_{l_1}(x^{(1)})\cdots H_{l_p}(x^{(p)})~~~~l_1,\ldots,l_p\in\NN,
\end{equation*}
where $H_{l_i}$ is an ordinary one-dimensional Hermite polynomial. Therefore, for any $x\in\RR^q$ the following series expansion holds in $L^2$
\begin{align}
1_{\{Y_j\leq x\}}-F(x)=&\sum_{k=1}^{\infty}\sum_{l_1+\ldots+l_p=k}\frac{J_{l_1,\ldots,l_p}(x)}{l_1!\cdots l_p!}H_{l_1,\ldots,l_p}(X_j),\label{multidarstellung}\\
\intertext{where}
J_{l_1,\ldots,l_p}(x)=&E(1_{\{G(X_j)\leq x\}}H_{l_1,\ldots,l_p}(X_j)).\numberthis\label{hc2}
\end{align}
The maximum index $m$ satisfying $J_{l_1,\ldots,l_p}(x)=0$ whenever $l_1+\ldots+l_p<m$ is called the Hermite rank of $1_{\{G(\cdot)\leq x\}}-F(x)$ and the Hermite rank of $(1_{\{G(\cdot)\leq x\}}-F(x))_{x\in\RR^q}$ is defined as the minimum of all pointwise Hermite ranks.

With respect to \eqref{multidarstellung} one recognizes that a limiting process for  $(R_N(x,t))$ in the context of mutivariate subordination differs from the one in section 2, since there could be up to $m+p-1\choose m$ multivariate Hermite polynomials contributing to the limit.\\
Let $(B^{(1)}),\ldots,B^{(p)})$ the joint random spectral measure satisfying
\begin{gather*}
\left\{d_N^{-1}\sum_{j=1}^{\lfloor Nt\rfloor}\left(H_m(X_j^{(1)}),\ldots,H_m(X_j^{(p)})\right):t\in[0,1]\right\}\xrightarrow{~~d~~}\\
\left\{\left(Z_m^{(1)}(t),\ldots,Z_m^{(p)}(t)\right):t\in[0,1]\right\},\numberthis\label{Hermite gemeinsam}
\end{gather*}
where $Z_m^{(k)}(t)$ is defined as in \eqref{Spektraldarstellung} with $B=B^{(k)}$ for $1\leq k \leq p$. 
\begin{Satz}\label{non central mv}
 Let $(X_j)_{j\geq1}$ be a $p$-dimensional Gaussian process satisfying \eqref{a1}, \eqref{a2} and \eqref{a3} and let 
 $G:\RR^p\rightarrow\RR^q$ be a measurable function. Furthermore, let $0<D<1/m$, where $m$ is the Hermite rank of $(1_{\{G(\cdot)\leq x\}}-F(x))_{x\in\RR^q}$. Then 
 \begin{equation*}
 \left\{d_N^{-1}R_N(x,t): (x,t)\in[-\infty,\infty]^q\times [0,1]\right\}
\end{equation*}
converges weakly in $D([-\infty,\infty]^q\times [0,1])$, to
 \begin{multline}\label{Grenzprozess}
 \Biggl\{c\sum_{j_1,\ldots,j_m=1}^p \tilde{J}_{j_1,\ldots,j_m}(x)\int_{\RR^m}^{''}\frac{e^{it(x_1+\ldots+x_m)}-1}{i(x_1+\ldots+x_m)}\prod_{j=1}^m|x_j|^{-(1-D)/2}\\B^{(j_1)}(dx_1)\cdots B^{(j_m)}(dx_m): 
(x,t)\in[-\infty,\infty]^q\times [0,1]\Biggr\} .
\end{multline}
The constant $c$ is the same as in \eqref{Spektraldarstellung} and
\begin{equation*}
\tilde{J}_{j_1,\ldots,j_m}(x)=(m!)^{-1}E\left(1_{\{G(X_1)\leq x\}}\prod_{i=1}^p H_{i(j_1,\ldots,j_m)}(X_1^{(j)})\right),
\end{equation*}
where $i(j_1,\ldots,j_m)$ is the number of indeces $j_1,\ldots,j_m$ that are eqal to $i$.
\end{Satz}
\begin{Remark}
\citet[Theorem 6]{Arcones} studied the generalized Hermite process \eqref{Grenzprozess} in the non-uniform case. Note that we have corrected the domain of integration, i.e. we have $\RR^m$ instead of $[-\pi,\pi]^m$. For further details on stochastic integrals with dependent integrators, see \citet{Fox}.
\end{Remark}
\begin{Remark}
Theorem \ref{non central mv} is not a corollary of Theorem 9 by \citet{Arcones}. Although Arcones states a non-central limit theorem for the empirical process indexed by functions, the class $\{1_{\{G(\cdot)\leq x\}}-f(x):x\in\RR^p\}$ does not satisfy the required bracketing condition.
\end{Remark}
\begin{Satz}[Reduction principle]\label{reduction mv}
 Let $(X_j)_{j\geq1}$ be a $p$-dimensional Gaussian process satisfying \eqref{a1}, \eqref{a2} and \eqref{a3} and let 
 $G:\RR^p\rightarrow\RR^q$ be a measurable function. Furthermore, let $0<D<1/m$, where $m$ is the Hermite rank of $(1_{\{G(\cdot)\leq x\}}-F(x))_{x\in\RR^q}$. Then there exist constants $C,\kappa>0$ such that for any $0<\varepsilon\leq1$
\begin{align}
 P\Biggl(&\max_{n\leq N}\sup_{x\in[-\infty,\infty]^q}d_N^{-1}\left|\sum_{j=1}^n\left(1_{\{Y_k\leq x\}}-F(x)-
 \sum_{l_1+\ldots+l_p=m}\frac{J_{l_1,\ldots,l_p}(x)}{l_1!\cdots l_p!}H_{l_1,\ldots,l_p}(X_j)\right)\right|
         >\varepsilon\Biggr)\notag\\
\leq &CN^{-\kappa}(1+\varepsilon^{-3}).\notag
\end{align}
The normalization factor is given by
\begin{equation*}
 d_N^2=\operatorname{Var}\left(\sum_{j=1}^N H_m(X_j^{(1)})\right).
\end{equation*}
\end{Satz}
\section{Proofs}\label{Beweise 1d}
We start by introducing suitable partitions of $\RR^p$. For $x\in\RR^p$ let 
\begin{equation}\label{Lambda}
 \Lambda(x):=\sum_{j=1}^p\Lambda_j(x^{(j)}),
\end{equation}
where $\Lambda_j:[-\infty,\infty]\rightarrow\RR$ are non-decreasing, right-continious functions satisfying $\Lambda_j(-\infty)=0$ and $\Lambda_j(\infty)=\Lambda_1(\infty)<\infty$ for all $j=1,\ldots,p$. These functions will be specified in section \ref{Proof 1 und 2} and \ref{Proof 3 und 4}, respectively.
For any $k\in\NN$, $1\leq i \leq 2^k-1$ and $1\leq j \leq p$ let
\begin{align}
 x_{i}^{(j)}(k):=&\inf\{x\in\RR:\Lambda_j(x)\geq\Lambda_1(\infty)i2^{-k}\},\notag\\
 x_{0}^{(j)}(k):=&-\infty,\notag\\
 x_{2^k}^{(j)}(k):=& \infty.\label{cp}
\end{align}
Furthermore, let $x_{0}^{(j)}(0)=-\infty$ and $x_{1}^{(j)}(0)=\infty$. Note that 
\begin{equation}\label{abstand}
 \Lambda_j(x_{i+1}^{(j)}(k)-)-\Lambda_j(x_{i}^{(j)}(k))\leq \Lambda_1(\infty)2^{-k}.
\end{equation}
Each partition will consist of disjoint boxes, whose vertices are described by the upper coordinates. To simplify the identification of such a partition, we will classify these into different \textit{qualities}. Note that the construction below is similar to the one used by \citet{Marinucci}.
\begin{Def}\label{Definition1}
i) We denote by $\mathcal{A}_{k_1,\ldots,k_p}$ the partition of $\RR^p$ whose elements have the form
\begin{equation*}
(x_{i_1}^{(1)}(k_1),x_{i_1+1}^{(1)}(k_1)]\times\ldots\times (x_{i_p}^{(p)}(k_p),x_{i_p+1}^{(p)}(k_p)],
\end{equation*}
$0\leq i_j\leq 2^{k_j-1}$.\\
ii) We call a partition $\mathcal{A}_{k_1,\ldots,k_p}$ a \textit{partition of quality} $k$, if $\max_{1\leq j\leq p}k_j=k$.
\end{Def}
The number of partitions of quality $k$ can be calculated as
\begin{align}
&\binom{p}{1}k^{p-1}+\binom{p}{2}k^{p-2}+\ldots+\binom{p}{p}k^{0}\notag\\
=&(k+1)^p-k^p.\label{anzk}
\end{align}
By \eqref{anzk}, the total number of all partitions of quality less or equal $K$ is given by
\begin{equation}
\sum_{k=1}^K (k+1)^p-k^p=(K+1)^p-1.
\end{equation}
For simplicity we denote these $(K+1)^p-1$ partitions by $\mathcal{A}_1,\ldots,\mathcal{A}_{(K+1)^p-1}$. For $x\in\RR^p$ let 
\begin{align*}
a_{x}(K)&:=(x_{i_K(x^{(1)})}^{(1)}(K),\ldots,x_{i_K(x^{(p)})}^{(p)}(K))\\
b_{x}(K)&:=(x_{i_K(x^{(1)})+1}^{(1)}(K),\ldots,x_{i_K(x^{(p)})+1}^{(p)}(K)),
\end{align*}
where for $1\leq k\leq K$ and $1\leq j\leq p$ $i_k(x^{(j)})$ denotes those index satisfying
\begin{equation}\label{schachtelung}
x_{i_k(x^{(j)})}^{(j)}(k)\leq x^{(j)}\leq x_{i_k(x^{(j)})+1}^{(j)}(k).
\end{equation}
\begin{Lemma}\label{Lemma 1}
For each $x\in\RR^p$ there exist disjoint sets $A_l(x)\in\mathcal{A}_l$ such that
\begin{equation}\label{ver}
\bigcup\limits_{1\leq l\leq (K+1)^p-1} A_l(x)=\{y\in\RR^p: y\leq a_{x}(K)\}.
\end{equation}
\end{Lemma}
\begin{proof}
For $x^{(j)}$, $1\leq j\leq p$, choose $i_k(x^{(j)})$, $1\leq k\leq K$ as in \eqref{schachtelung}. This yields
\begin{equation*}
-\infty=x_{i_0(x^{(j)})}^{(j)}(0)\leq x_{i_1(x^{(j)})}^{(j)}(1)\leq\ldots\leq x_{i_K(x^{(j)})}^{(j)}(K)\leq x^{(j)}.
\end{equation*}
We define the sets $A_l(x)$  by
\begin{equation*}
\bigtimes_{j=1}^p (x_{i_{l_j}(x^{(j)})}^{(j)}(l_j),x_{i_{l_j+1}(x^{(j)})}^{(j)}(l_j+1)],~~~0\leq l_j\leq K-1.
\end{equation*}
These sets are disjoint since if there exist an element $y\in A_l(x)\cap A_k(x)$ one have 
$x_{i_{l_j}(x^{(j)})}^{(j)}(l_j)\leq y^{(j)}\leq x_{i_{l_j+1}(x^{(j)})}^{(j)}(l_j+1)$
and $x_{i_{k_j}(x^{(j)})}^{(j)}(k_j)\leq y^{(j)}\leq x_{i_{k_j+1}(x^{(j)})}^{(j)}(k_j+1)$. This implies $i_{l_j}(x^{(j)})=i_{k_j}(x^{(j)})$ and therefore $A_l(x)=A_k(x)$.
\end{proof}
\subsection{Proofs of Theorem 1 and 2}\label{Proof 1 und 2}
For any $x\in\RR^p$ and for any measurable $A\subset \RR^p$ we define 
\begin{align*}
 S_N(n,x)&:=d_N^{-1}\sum_{j=1}^n\left(1_{\{Y_j\leq x\}}-F(x)-\frac{J_m(x)}{m!}H_m(X_j)\right)\\
S_N(n,A)&:=d_N^{-1}\sum_{j=1}^n\left(1_{\{Y_j\in A\}}-P(Y_j\in A)-\frac{J_m(A)}{m!}H_m(X_j)\right).
\end{align*}
Furthermore, regarding \eqref{hc} we set
\begin{equation}
J_q(A):=E(1_{\{Y_j\in A\}}H_q(X_j))
\end{equation}
Since $J_m(A)+J_m(B)=J_m(A\cup B)$ for disjoint sets $A$ and $B$, we can use \eqref{ver} to get the following representation
\begin{equation}\label{darstellung}
S_N(n,x)=\sum_{l=1}^{(K+1)^p-1}S_N(n,A_l(x))+S_N(n,a_{x}(K),x).
\end{equation}
Lemma \ref{momente} will give us a second order moment bound for $S_N(n,A)$. It is due to Lemma 3.1. by \citet{Dehling}.
\begin{Lemma}\label{momente}
 There exist constants $\gamma>0$ and $C>0$ such that for all $A\subset\RR^p$, $n\leq N$
\begin{equation*}
 E\left|S_N(n,A)\right|^2\leq C\left(\frac{n}{N}\right)N^{-\gamma}P(Y_1\in A).
\end{equation*}
\end{Lemma}
\begin{proof}
 Since the Hermite polynomials form an orthogonal basis of $L_2(\varphi(x)dx)$, the representation
\begin{equation*}
 1_{\{Y_j\in A\}}-P(Y_1\in A)=\sum_{q=1}^{\infty}\frac{J_q(A)}{q!}H_q(X_j)
\end{equation*}
yields 
\begin{align*}
 &\sum_{q=1}^{\infty}\frac{J_q^2(A)}{q!}\\
=&E\left(1_{\{Y_j\in A\}}-P(Y_1\in A)\right)^2\\
=&(1-P(Y_1\in A))^2P(Y_1\in A)+P(Y_1\in A)^2(1-P(Y_1\in A)))\\
\leq&P(Y_1\in A).
\end{align*}
Now along the lines of Lemma 3.1. by \citet{Dehling} we get
\begin{align*}
 E\left|S_N(n,A)\right|^2\leq&P(Y_1\in A)d_N^{-2}\sum_{j,k\leq n}|r(j-k)|^{m+1}\\
\leq&CP(Y_1\in A)\left(\frac{n}{N}\right)N^{mD-1\vee -D}L(n)L(N)^{-m},
\end{align*}
which completes the proof. 
\end{proof}
For $\Lambda(x)$ defined by \eqref{Lambda} let
\begin{equation*}
 \Lambda_j(x^{(j)}):=F_j(x^{(j)})+\int\limits_{\{G_j(s)\leq x^{(j)}\}}\frac{|H_m(s)|}{m!}\varphi(s)ds.
\end{equation*}
\begin{Lemma}\label{Lemma 3}
The increments $(m!)^{-1}J_m(x,y):=(m!)^{-1}(J_m(y)-J_m(x))$ and $F(x,y):=F(y)-F(x)$ are bounded by $\Lambda(y)-\Lambda(x)$ for all $x\leq y$.
\end{Lemma}
\begin{proof}
 Let $x\leq y$ and set $B:=\{s\in\RR:G(s)\leq y,G(s)\not\leq x\}$ and $B_j:=\{s\in\RR:x^{(j)}\leq G_j(s)\leq y^{(j)}\}$. Since $B\subset\bigcup\limits_{j=1}^pB_j$
we have
\begin{align*}
 J_m(y)-J_m(x)=&\int\limits_{\{G(s)\leq y\}}H_m(s)\varphi(s)ds-\int_{\{G(s)\leq x\}}H_m(s)\varphi(s)ds\\
                =&\int\limits_B H_m(s)\varphi(s)ds\\
             \leq&\int\limits_B |H_m(s)|\varphi(s)ds\\
             \leq&\sum_{j=1}^p\int\limits_{B_j} |H_m(s)|\varphi(s)ds\\
             \leq&\sum_{j=1}^p\biggl(\int\limits_{\{G_j(s)\leq y^{(j)}\}}|H_m(s)|\varphi(s)ds-\int_{\{G_j(s)\leq x^{(j)}\}}|H_m(s)|\varphi(s)ds\biggr)\\
             \leq&\sum_{j=1}^p(\Lambda_j(y^{(j)})-\Lambda_j(x^{(j)}))\\
                =&\Lambda(y)-\Lambda(x).
\end{align*}
Moreover,
\begin{align*}
 F(y)-F(x)\leq&\sum_{j=1}^p(F_j(y^{(j)})-F_j(x^{(j)}))\\
            \leq&\sum_{j=1}^p(\Lambda_j(y^{(j)})-\Lambda_j(x^{(j)}))\\
               =&\Lambda(y)-\Lambda(x).
\end{align*}
\end{proof}
\begin{Lemma}
 There exist constants $\rho, C>0$ such that for all $n\leq N$ and $0<\varepsilon\leq1$
\begin{equation*}
 P\left(\sup\limits_{x\in\RR^p}\left|S_N(n,x)\right|>\varepsilon\right)\leq CN^{-\rho}\left(\left(\frac{n}{N}\right)\varepsilon^{-3}
 +\left(\frac{n}{N}\right)^{2-mD}\right).
\end{equation*}
\end{Lemma}

\begin{proof}
 Since we want to use representation \eqref{darstellung}, we will start by bounding $|S_N(n,a_{x}(K),x)|$.
\begin{align*}
 &|S_N(n,a_{x}(K),x)|\\
=&\Biggl|d_N^{-1}\sum_{j=1}^n\Bigl(\bigl(1_{\{Y_j\leq x\}}-1_{\{Y_j\leq
a_{x}(K)\}}\bigr)-F(a_{x}(K),x)-\frac{1}{m!}J_m(a_{x}(K),x)H_m(X_j)\Bigr)\Biggr|\\
\leq&d_N^{-1}\sum_{j=1}^n\Bigl(\bigl(1_{\{Y_j< b_{x}(K)\}}-1_{\{Y_j\leq a_{x}(K)\}}\bigr)+F(a_{x}(K),b_{x}(K)-)\Bigr)\\
&~~~~+\frac{1}{m!}J_m(a_{x}(K),b_{x}(K)-)d_N^{-1}\Bigl|\sum_{j=1}^nH_m(X_j)\Bigr|\\
\leq&|S_N(n,a_{x}(K),b_{x}(K)-)|+2nd_N^{-1}F(a_{x}(K),b_{x}(K)-)\\
&~~~~+\frac{2}{m!}d_N^{-1}J_m(a_{x}(K),b_{x}(K)-)\Bigl|\sum_{j=1}^nH_m(X_j)\Bigr|\\
\leq&|S_N(n,a_{x}(K),b_{x}(K)-)|+2nd_N^{-1}\Lambda(a_{x}(K),b_{x}(K)-)\\
&~~~~+2d_N^{-1}\Lambda(a_{x}(K),b_{x}(K)-)\Bigl|\sum_{j=1}^nH_m(X_j)\Bigr|\\
\leq&|S_N(n,a_{x}(K),b_{x}(K)-)|+2nd_N^{-1}p\Lambda_1(\infty)2^{-K}+2d_N^{-1}p\Lambda_1(\infty)2^{-K}\Bigl|\sum_{j=1}^nH_m(X_j)\Bigr|\numberthis\label{eps/2}
\end{align*}
By using \eqref{darstellung}, \eqref{eps/2} and $\sum_{l=1}^{\infty}\varepsilon/(l+4)^2<\varepsilon/4$ we get
\begin{align}
 &P\left(\sup\limits_{x\in\RR^p}\left|S_N(n,x)\right|>\varepsilon\right)\notag\\
\leq\sum_{l=1}^{(K+1)^p-1}&P\left(\max\limits_{x\in\RR^p}\left|S_N(n,A_l(x))\right|>\varepsilon/(l+4)^2\right)\notag\\
+&P\left(\sup\limits_{x\in\RR^p}\left|S_N(n,a_{x}(K),b_{x}(K)-)\right|>\varepsilon/4\right)\notag\\
+&P\biggl(2d_N^{-1}p\Lambda_1(\infty)2^{-K}\Bigl|\sum_{j=1}^nH_m(X_j)\Bigr|>\varepsilon/2-2nd_N^{-1}p\Lambda_1(\infty)2^{-K}\biggr).\label{3.1}
\end{align}
Remember, the elements $A_l(1),\ldots A_l(|\mathcal{A}_l|)$ of $\mathcal{A}_l$ are disjoint. Therefore Lemma \ref{momente} yields
\begin{align*}
 &P\left(\max\limits_{x\in\RR^p}\left|S_N(n,A_l(x))\right|>\varepsilon/(l+4)^2\right)\\
\leq&\sum_{i=1}^{|\mathcal{A}_l|}P\left(\left|S_N(n,A_l(i))\right|>\varepsilon/(l+4)^2\right)\\
\leq&\sum_{i=1}^{|\mathcal{A}_l|}(l+4)^4\varepsilon^{-2}E\left|S_N(n,A_l(i))\right|^2\\
\leq&C\left(\frac{n}{N}\right)N^{-\gamma}(l+4)^4\varepsilon^{-2}\sum_{i=1}^{|\mathcal{A}_l|}P(Y_1\in A_l(i))\\
=&C\left(\frac{n}{N}\right)N^{-\gamma}(l+4)^4\varepsilon^{-2},\numberthis\label{3.2}
\end{align*}
for $1\leq l\leq (K+1)^p-1$. The next-to-last summand can be bounded as follows. Similar to \eqref{ver} and \eqref{darstellung} each partition of quality $K$ contains one element $B_l(x)$ such that 
\begin{equation*}
S_N(n,a_{x}(K),b_{x}(K)-)=\sum_{l=1}^{(K+1)^p-K^p} S_N(n,B_l(x)-).
\end{equation*}
With respect to \eqref{3.2} we get
\begin{align*}
 &P\left(\sup\limits_{x\in\RR^p}\left|S_N(n,a_{x}(K),b_{x}(K)-)\right|>\varepsilon/4\right)\\
\leq&\sum_{l=1}^{(K+1)^p-K^p}P\left(\max\limits_{x\in\RR^p}\left|S_N(n,B_l(x)-)\right|>\varepsilon/(4(l+4)^2)\right)\\
\leq&C\left(\frac{n}{N}\right)N^{-\gamma}\varepsilon^{-2}\sum_{l=1}^{(K+1)^p-K^p}(l+4)^4.\numberthis\label{3.3}
\end{align*}
Now let
\begin{equation*}
 K=\left\lceil\log_2\left(\frac{8p\Lambda_1(\infty)}{\varepsilon}Nd_N^{-1}\right)\right\rceil.
\end{equation*}
This choice implies
\begin{align*}
 &2Nd_N^{-1}p\Lambda_1(\infty)2^{-K}\leq\frac{\varepsilon}{4}\\
 &\left(\frac{\varepsilon}{4}\right)^{-2}\leq N^{-2}d_N^2(p\Lambda_1(\infty))^{-2}2^{2K-2}
\end{align*}
and therefore
\begin{align}
 &P\biggl(2d_N^{-1}p\Lambda_1(\infty)2^{-K}\Bigl|\sum_{j=1}^nH_m(X_j)\Bigr|>\frac{\varepsilon}{2}-2nd_N^{-1}p\Lambda_1(\infty)2^{-K}\biggr)\notag\\
\leq&P\biggl(2d_N^{-1}p\Lambda_1(\infty)2^{-K}\Bigl|\sum_{j=1}^nH_m(X_j)\Bigr|>\frac{\varepsilon}{4}\biggr)\notag\\
\leq&P\biggl(d_N^{-1}\Bigl|\sum_{j=1}^nH_m(X_j)\Bigr|>\frac{\varepsilon}{4}\cdot\frac{2^{K-1}}{p\Lambda_1(\infty)}\biggr)\notag\\
\leq&d_N^{-2}E\Bigl|\sum_{j=1}^nH_m(X_j)\Bigr|^2\left(\frac{\varepsilon}{4}\right)^{-2}2^{-2K+2}(p\Lambda_1(\infty))^2\notag\\
\leq&C\left(\frac{d_n}{d_N}\right)^2N^{-2}d_N^2\notag\\
\leq&C\left(\frac{n}{N}\right)^{2-mD}\left(\frac{L(n)}{L(N)}\right)^mN^{-mD}L^m(N)\notag\\
\leq&C\left(\frac{n}{N}\right)^{2-mD}N^{-mD+\lambda},\label{3.4}
\end{align}
for any $\lambda>0$. By using \eqref{3.1}, \eqref{3.2}, \eqref{3.3}, \eqref{3.4} we get 
\begin{align*}
 &P\left(\sup\limits_{x\in\RR^p}\left|S_N(n,x)\right|>\varepsilon\right)\\
\leq&C\left(\frac{n}{N}\right)N^{-\gamma}\varepsilon^{-2}\left(\sum_{l=1}^{(K+1)^p-1}(l+4)^4+\sum_{l=1}^{(K+1)^p-K^p}(l+4)^4\right)
 +C\left(\frac{n}{N}\right)^{2-mD}N^{-mD+\lambda}\\
\leq&C\left(\frac{n}{N}\right)N^{-\gamma}\varepsilon^{-2}K^{5p}+C\left(\frac{n}{N}\right)^{2-mD}N^{-mD+\lambda}\numberthis\label{vorletzte}
\end{align*}
Since 
\begin{align*}
K^{5p}\leq&C\left(\log(\varepsilon^{-1})^{5p}+\log(N)^{5p}\right)\\
           \leq&C\varepsilon^{-1}N^{\delta}
\end{align*}
for any $\delta>0$, \eqref{vorletzte} is bounded by
 \begin{equation*}
 CN^{(-\gamma+\delta)\vee(-mD+\lambda)}\left(\left(\frac{n}{N}\right)\varepsilon^{-3} +\left(\frac{n}{N}\right)^{2-mD}\right),
 \end{equation*}
 which completes the proof.
\end{proof}
To prove Theorem 1 and 2 we can use the proofs which were given by \citet{Dehling} in the case of one dimensional observations.
\subsection{Proofs of Theorem 3 and 4}\label{Proof 3 und 4}
For simplicity we assume that $p=q$, i.e. $G:\RR^p\rightarrow\RR^p$. Since the main idea for proving Theorem \ref{reduction mv} was already used in the proof of Theorem \ref{rp}. Therefore, we will use some modified definitions and notations from section \ref{Proof 1 und 2}. From now on let
\begin{align*}
 S_N(n,x)&:=d_N^{-1}\sum_{j=1}^n\left(1_{\{Y_j\leq x\}}-F(x)-\sum_{l_1+\ldots+l_p=m}\frac{J_{l_1,\ldots,l_p}(x)}{l_1!\cdots l_p!}H_{l_1,\ldots,l_p}(X_j)\right),\\
 S_N(n,A)&:=d_N^{-1}\sum_{j=1}^n\left(1_{\{Y_j\in A\}}-P(Y_j\in A)-\sum_{l_1+\ldots+l_p=m}\frac{J_{l_1,\ldots,l_p}(A)}{l_1!\cdots l_p!}H_{l_1,\ldots,l_p}(X_j)\right),\\
 J_{l_1,\ldots,l_p}(A)&:=E(1_{\{Y_j\in A\}}H_{l_1,\ldots,l_p}(X_j)).
\end{align*}
For $\Lambda(x)$ defined in \eqref{Lambda} replace the terms of the sum by
\begin{equation*}
 \Lambda_j(x^{(j)}):=P(G_j(X_1)\leq x^{(j)})+\sum_{l_1+\ldots+l_p=m}\int\limits_{\substack{\{s\in\RR^p:\\G_j(s)\leq x^{(j)}\}}}\frac{|H_{l_1,\ldots,l_p}(s)|}{l_1!\cdots l_p!}\varphi(s)ds^{(1)}\ldots ds^{(p)},
\end{equation*}
where $\varphi$ denotes the $p$-dimensional standard normal distribution, and define the chaining points $x_i^{(j)}(k)$ analogous to \eqref{cp}. The partitions $\mathcal{A}_{k_1,\ldots,k_p}$ are given as in \defref{Definition1}. Therefore, Lemma \ref{Lemma 1} and representation \eqref{darstellung} still hold, i.e.
\begin{equation}\label{darstellung2}
S_N(n,x)=\sum_{l=1}^{(K+1)^p-1}S_N(n,A_l(x))+S_N(n,a_{x}(K),x).
\end{equation}
\begin{Lemma}\label{Lemma 3 mv}
The increments $J_m(x,y)$ and $F(x,y)$, where
\begin{align*}
J_m(x,y):=&\sum_{l_1+\ldots+l_p=m}\frac{J_{l_1,\ldots,l_p}(y)-J_{l_1,\ldots,l_p}(x)}{l_1!\cdots l_p!}\\
F(x,y):=&F(y)-F(x),
\end{align*}
are both bounded by $\Lambda(x,y):=\Lambda(y)-\Lambda(x)$ for all $x\leq y$.
\end{Lemma}
Lemma \ref{Lemma 3 mv} can be proven in the same way as Lemma \ref{Lemma 3}. Lemma \ref{LemmaArcones} is due to \citet[Lemma 1]{Arcones}.
\begin{Lemma}\label{LemmaArcones}
Let $X=(X^{(1)},\ldots,X^{(p)})$ and $Y=(Y^{(1)},\ldots,Y^{(p)})$ be two mean-zero Gaussian random vectors on $\RR^p$. Assume that
\begin{equation}\label{a4}
EX^{(i)}X^{(j)}=EY^{(i)}Y^{(j)}=\delta_{i,j}
\end{equation}
for each $1\leq i,j\leq p$. We define 
\begin{equation*}
r^{(i,j)}:=EX^{(i)}Y^{(j)}.
\end{equation*}
Let $f$ be a function on $\RR^p$ with finite second moment and Hermite rank $m$, $1\leq m<\infty$, with respect to $X$. Suppose that
\begin{equation*}
\psi:=\left(\sup_{1\leq i\leq p}\sum_{j=1}^p|r^{(i,j)}|\right)\vee\left(\sup_{1\leq j\leq p}\sum_{i=1}^p|r^{(i,j)}|\right)\leq 1.
\end{equation*}
Then 
\begin{equation*}
\left|E(f(X)-Ef(X))(f(Y)-Ef(Y))\right|\leq\psi^mEf(X)^2.
\end{equation*}
\end{Lemma}

\begin{Lemma}\label{momente2}
There exist constants $\gamma>0$ and $C>0$ such that for all measurable $A\subset\RR^p$ and $n\leq N$
\begin{equation*}
E|S_N(n,A)|^2\leq C\left(\frac{n}{N}\right)N^{-\gamma}P(Y_1\in A).
\end{equation*}
\end{Lemma}
\begin{proof}
Let
\begin{align*}
f(\cdot):=&1_{\{G(\cdot)\in A\}}-P(Y_1\in A)-\sum_{l_1+\ldots+l_p=m}\frac{J_{l_1,\ldots,l_p}(A)}{l_1!\cdots l_p!}H_{l_1,\ldots,l_p}(\cdot)\\
\psi(k):=&\left(\max_{1\leq i\leq p}\sum_{j=1}^p|r^{(i,j)}(k)|\right)\vee\left(\max_{1\leq j\leq p}\sum_{i=1}^p|r^{(i,j)}(k)|\right).
\end{align*}
Since $r^{(i,j)}(k)$ converges to $0$ for all $1\leq i,j\leq p$ if $k$ tends to infinity, we can find $b\in\NN$ such that $\psi(kb)\leq 1$ for all $k\geq 1$.  Note that $f$ has Hermite rank $m+1$. Therefore, as \citet[p. 2249]{Arcones} we can apply Lemma  \ref{LemmaArcones} as follows
\begin{align*}
&E\left|\sum_{j=1}^nf(X_j)\right|^2\\
=&E\left|\sum_{j=1}^b\sum_{k=1}^{\lfloor n-j+b/b\rfloor}f(X_{(k-1)b+j})\right|^2\\
\leq&C\sum_{j=1}^b\sum_{k,l=1}^{\lfloor n-j+b/b\rfloor}E\left(f(X_{(k-1)b+j})f(X_{(l-1)b+j})\right)\\
\leq&C\sum_{j=1}^b\sum_{k,l=1}^{\lfloor n-j+b/b\rfloor}\psi(b(k-l))^{m+1}Ef(X_1)^2\\
\leq&CEf(X_1)^2\sum_{k,l=1}^{n}\psi(b(k-l))^{m+1}\\
\end{align*}
By using
\begin{align*}
&Ef(X_1)^2\\
=&\sum_{\substack{k_1+\ldots+k_p\geq m+1\\l_1+\ldots+l_p\geq m+1}}J_{k_1,\ldots,k_p}J_{l_1,\ldots,l_p}\prod_{j=1}^p(k_j!l_j!)^{-1}\\
&{}\hspace*{30mm} \cdot E\Bigl(H_{k_1}(X_1^{(1)})\cdots H_{k_p}(X_1^{(p)})H_{l_1}(X_1^{(1)})\cdots H_{l_p}(X_1^{(p)})\Bigr)\\
=&\sum_{k_1+\ldots+k_p\geq m+1}J_{k_1,\ldots,k_p}^2\prod_{j=1}^p(k_j!)^{-2}
E\left(H_{k_1}(X_1^{(1)})\right)^2\cdots E\left(H_{k_p}(X_1^{(p)})\right)^2\\
\leq&\sum_{k_1,\ldots,k_p=0}J_{k_1,\ldots,k_p}^2\prod_{j=1}^p(k_j!)^{-1}\\
=&E\left(1_{\{G(X_1)\in A\}}-P(Y_1\in A)\right)^2\\
\leq&P(Y_1\in A)
\end{align*}
and
\begin{align*}
&\psi(bk)\\
\leq&\sum_{i,j=1}^p\left|r^{(i,j)}(bk)\right|\\
=&\sum_{i,j=1}^p\left|c_{(i,j)}L(bk)(bk)^{-D}\right|\\
\leq& C\left|L'(k)k^{-D}\right|,
\end{align*} 
where $L'(k):=L(bk)$ is slowly varying at infinity, we get
\begin{equation}\label{momentesumme}
E\left|\sum_{j=1}^nf(X_j)\right|^2\leq CP(Y_1\in A)\sum_{k,l=1}^{n}\left|L'(k-l)(k-l)^{-D}\right|^{m+1}.
\end{equation}

The remaining parts of the proof can be found in \citet[p. 1777]{Dehling}
\end{proof}
\begin{Lemma}\label{asymptotischevar}
For all $m\in\NN$ there exists a constant $C>0$ such that for all $l_1,\ldots,l_p\in\NN$, $l_1+\ldots+l_p=m$, and $n\in\NN$
\begin{equation*}
E\left|\sum_{j=1}^nH_{l_1,\ldots,l_p}(X_j)\right|^2\leq Cd_n^2
\end{equation*}
\end{Lemma}
\begin{proof}
By using \eqref{momentesumme} with $f=H_{l_1,\ldots,l_p}$ we get
\begin{align*}
&E\left|\sum_{j=1}^nH_{l_1,\ldots,l_p}(X_j)\right|^2\\
\leq&CEH_{l_1,\ldots,l_p}(X_1)^2\sum_{k,l=1}^{n}\left|L(b(k-l))^m(k-l)^{-mD}\right|\\
\leq&Cn\sum_{k=1}^{n}\left|L(bk)^mk^{-mD}\right|\\
\leq&Cn^{2-mD}L(bn)^m\\
\leq&Cd_n^2.
\end{align*}
\end{proof}
\begin{Lemma}\label{Dehling3.2}
 There exist constants $\rho, C>0$ such that for all $n\leq N$ and $0<\varepsilon\leq1$
\begin{equation*}
 P\left(\sup\limits_{x\in\RR^p}\left|S_N(n,x)\right|>\varepsilon\right)\leq CN^{-\rho}\left(\left(\frac{n}{N}\right)\varepsilon^{-3}
 +\left(\frac{n}{N}\right)^{2-mD}\right).
\end{equation*}
\end{Lemma}
\begin{proof}
 Since we want to use representation \eqref{darstellung2}, we will start by bounding $|S_N(n,a_{x}(K),x)|$.
\begin{align*}
 &|S_N(n,a_{x}(K),x)|\\
=&\Biggl|d_N^{-1}\sum_{j=1}^n\Bigl(\bigl(1_{\{Y_j\leq x\}}-1_{\{Y_j\leq
a_{x}(K)\}}\bigr)-F(a_{x}(K),x)-\sum_{l_1+\ldots+l_p=m}\frac{J_{l_1,\ldots,l_p}(a_{x}(K),x)}{l_1!\cdots l_p!}H_{l_1,\ldots,l_p}(X_j)\Bigr)\Biggr|\\
\leq&d_N^{-1}\sum_{j=1}^n\Bigl(\bigl(1_{\{Y_j< b_{x}(K)\}}-1_{\{Y_j\leq a_{x}(K)\}}\bigr)+F(a_{x}(K),b_{x}(K)-)\Bigr)\\
&~~~~+\sum_{l_1+\ldots+l_p=m}\frac{J_{l_1,\ldots,l_p}(a_{x}(K),b_{x}(K)-)}{l_1!\cdots l_p!}d_N^{-1}\Bigl|\sum_{j=1}^nH_{l_1,\ldots,l_p}(X_j)\Bigr|\\
\leq&|S_N(n,a_{x}(K),b_{x}(K)-)|+2nd_N^{-1}F(a_{x}(K),b_{x}(K)-)\\
&~~~~+2d_N^{-1}\sum_{l_1+\ldots+l_p=m}\frac{J_{l_1,\ldots,l_p}(a_{x}(K),b_{x}(K)-)}{l_1!\cdots l_p!}\Bigl|\sum_{j=1}^nH_{l_1,\ldots,l_p}(X_j)\Bigr|\\
\leq&|S_N(n,a_{x}(K),b_{x}(K)-)|+2nd_N^{-1}\Lambda(a_{x}(K),b_{x}(K)-)\\
&~~~~+2d_N^{-1}\Lambda(a_{x}(K),b_{x}(K)-)\sum_{l_1+\ldots+l_p=m}\Bigl|\sum_{j=1}^nH_{l_1,\ldots,l_p}(X_j)\Bigr|\\
\leq&|S_N(n,a_{x}(K),b_{x}(K)-)|+2nd_N^{-1}p\Lambda_1(\infty)2^{-K}\\
&~~~~+2d_N^{-1}p\Lambda_1(\infty)2^{-K}\sum_{l_1+\ldots+l_p=m}\Bigl|\sum_{j=1}^nH_{l_1,\ldots,l_p}(X_j)\Bigr|\numberthis\label{eps/2_2}
\end{align*}
By using \eqref{darstellung2}, \eqref{eps/2_2} and $\sum_{l=1}^{\infty}\varepsilon/(l+4)^2<\varepsilon/4$ we get
\begin{align}
 &P\left(\sup\limits_{x\in\RR^p}\left|S_N(n,x)\right|>\varepsilon\right)\notag\\
\leq\sum_{l=1}^{(K+1)^p-1}&P\left(\max\limits_{x\in\RR^p}\left|S_N(n,A_l(x))\right|>\varepsilon/(l+4)^2\right)\notag\\
+&P\left(\sup\limits_{x\in\RR^p}\left|S_N(n,a_{x}(K),b_{x}(K)-)\right|>\varepsilon/4\right)\notag\\
+&P\biggl(2d_N^{-1}p\Lambda_1(\infty)2^{-K}\sum_{l_1+\ldots+l_p=m}\Bigl|\sum_{j=1}^nH_{l_1,\ldots,l_p}(X_j)\Bigr|>\varepsilon/2-2nd_N^{-1}p\Lambda_1(\infty)2^{-K}\biggr).\label{3.1_2}
\end{align}
Remember, the elements $A_l(1),\ldots A_l(|\mathcal{A}_l|)$ of $\mathcal{A}_l$ are disjoint. Therefore Lemma \ref{momente2} yields
\begin{align*}
 &P\left(\max\limits_{x\in\RR^p}\left|S_N(n,A_l(x))\right|>\varepsilon/(l+4)^2\right)\\
\leq&\sum_{i=1}^{|\mathcal{A}_l|}P\left(\left|S_N(n,A_l(i))\right|>\varepsilon/(l+4)^2\right)\\
\leq&\sum_{i=1}^{|\mathcal{A}_l|}(l+4)^4\varepsilon^{-2}E\left|S_N(n,A_l(i))\right|^2\\
\leq&C\left(\frac{n}{N}\right)N^{-\gamma}(l+4)^4\varepsilon^{-2}\sum_{i=1}^{|\mathcal{A}_l|}P(Y_1\in A_l(i))\\
=&C\left(\frac{n}{N}\right)N^{-\gamma}(l+4)^4\varepsilon^{-2},\numberthis\label{3.2_2}
\end{align*}
for $1\leq l\leq (K+1)^p-1$. The next-to-last summand in \eqref{3.1_2} can be bounded as follows. Similar to \eqref{ver} and \eqref{darstellung2} each partition of quality $K$ contains one element $B_l(x)$ such that 
\begin{equation*}
S_N(n,a_{x}(K),b_{x}(K)-)=\sum_{l=1}^{(K+1)^p-K^p} S_N(n,B_l(x)-).
\end{equation*}
With respect to \eqref{3.2_2} we get
\begin{align*}
 &P\left(\sup\limits_{x\in\RR^p}\left|S_N(n,a_{x}(K),b_{x}(K)-)\right|>\varepsilon/4\right)\\
\leq&\sum_{l=1}^{(K+1)^p-K^p}P\left(\max\limits_{x\in\RR^p}\left|S_N(n,B_l(x)-)\right|>\varepsilon/(4(l+4)^2)\right)\\
\leq&C\left(\frac{n}{N}\right)N^{-\gamma}\varepsilon^{-2}\sum_{l=1}^{(K+1)^p-K^p}(l+4)^4.\numberthis\label{3.3_2}
\end{align*}
Now let
\begin{align*}
M=&\left|\left\{(l_1,\ldots,l_p)\in\NN^p:l_1+\ldots l_p=m\right\}\right|\\
\intertext{and}
 K=&\left\lceil\log_2\left(\frac{8p\Lambda_1(\infty)}{\varepsilon}Nd_N^{-1}\right)\right\rceil.
\end{align*}
This choice implies
\begin{align*}
 &2Nd_N^{-1}p\Lambda_1(\infty)2^{-K}\leq\frac{\varepsilon}{4}\\
 &\left(\frac{\varepsilon}{4}\right)^{-2}\leq N^{-2}d_N^2(p\Lambda_1(\infty))^{-2}2^{2K-2}
\end{align*}
and together with Lemma \ref{asymptotischevar} we obtain
\begin{align}
 &P\biggl(2d_N^{-1}p\Lambda_1(\infty)2^{-K}\sum_{l_1+\ldots+l_p=m}\Bigl|\sum_{j=1}^nH_{l_1,\ldots,l_p}(X_j)\Bigr|>\varepsilon/2-2nd_N^{-1}p\Lambda_1(\infty)2^{-K}\biggr)\notag\\
\leq&P\biggl(2d_N^{-1}p\Lambda_1(\infty)2^{-K}\sum_{l_1+\ldots+l_p=m}\Bigl|\sum_{j=1}^nH_{l_1,\ldots,l_p}(X_j)\Bigr|>\varepsilon/4\biggr)\notag\\
\leq&P\biggl(d_N^{-1}\sum_{l_1+\ldots+l_p=m}\Bigl|\sum_{j=1}^nH_{l_1,\ldots,l_p}(X_j)\Bigr|>\frac{\varepsilon}{4}\cdot\frac{2^{K-1}}{p\Lambda_1(\infty)}\biggr)\notag\\
\leq&\sum_{l_1+\ldots+l_p=m}P\biggl(d_N^{-1}\Bigl|\sum_{j=1}^nH_{l_1,\ldots,l_p}(X_j)\Bigr|>\frac{\varepsilon}{4}\cdot\frac{2^{K-1}}{p\Lambda_1(\infty)M}\biggr)\notag\\
\leq&d_N^{-2}\sum_{l_1+\ldots+l_p=m}E\Bigl|\sum_{j=1}^nH_{l_1,\ldots,l_p}(X_j)\Bigr|^2\left(\frac{\varepsilon}{4}\right)^{-2}2^{-2K+2}(p\Lambda_1(\infty))^2M^2\notag\\
\leq&C\left(\frac{d_n}{d_N}\right)^2N^{-2}d_N^2\notag\\
\leq&C\left(\frac{n}{N}\right)^{2-mD}\left(\frac{L(n)}{L(N)}\right)^mN^{-mD}L^m(N)\notag\\
\leq&C\left(\frac{n}{N}\right)^{2-mD}N^{-mD+\lambda},\label{3.4_2}
\end{align}
for any $\lambda>0$. By using \eqref{3.1_2}, \eqref{3.2_2}, \eqref{3.3_2}, \eqref{3.4_2} we get 
\begin{align*}
 &P\left(\sup\limits_{x\in\RR^p}\left|S_N(n,x)\right|>\varepsilon\right)\\
\leq&C\left(\frac{n}{N}\right)N^{-\gamma}\varepsilon^{-2}\left(\sum_{l=1}^{(K+1)^p-1}(l+4)^4+\sum_{l=1}^{(K+1)^p-K^p}(l+4)^4\right)
 +C\left(\frac{n}{N}\right)^{2-mD}N^{-mD+\lambda}\\
\leq&C\left(\frac{n}{N}\right)N^{-\gamma}\varepsilon^{-2}K^{5p}+C\left(\frac{n}{N}\right)^{2-mD}N^{-mD+\lambda}\numberthis\label{vorletzte2}
\end{align*}
Since 
\begin{align*}
K^{5p}\leq&C\left(\log(\varepsilon^{-1})^{5p}+\log(N)^{5p}\right)\\
           \leq&C\varepsilon^{-1}N^{\delta}
\end{align*}
for any $\delta>0$, \eqref{vorletzte2} is bounded by
 \begin{equation*}
 CN^{(-\gamma+\delta)\vee(-mD+\lambda)}\left(\left(\frac{n}{N}\right)\varepsilon^{-3} +\left(\frac{n}{N}\right)^{2-mD}\right),
 \end{equation*}
 which completes the proof.
\end{proof}
\begin{proof}[Proof of \thref{reduction mv}]
Lemma \ref{Dehling3.2} corresponds to Lemma 3.2 by \citet{Dehling}. Therefore the proof of Theorem \ref{reduction mv} is the same as in the one-dimensional case, see \citep[p.1781]{Dehling}.
\end{proof}
Instead of studying the partial sum process of multivariate Hermite polynomials we will deduce the asymptotics of $(R_N(x,t))$ from a linear combination of $m$th order univariate Hermite polynomials. Therefor we use the following Lemma.
\begin{Lemma}\label{LemmaRunge}
For all $m\in\NN$ and $a_1,\ldots,a_p\in\RR$ with $a_1^2+\ldots+a_p^2=1$ we have
\begin{equation}\label{Runge}
H_m\left(\sum_{j=1}^pa_jx_j\right)=\sum_{m_1+\ldots+m_p=m}\frac{m!}{m_1!\cdots m_p!}\prod_{j=1}^pa_j^{m_j}H_{m_j}(x_j).
\end{equation}
\end{Lemma}
Since we could not find a proof for this well known result in literature we give one here.
\begin{proof}
We first show that all partial derivatives are equal by using induction. For $m=1$ this is obvious. Remember that $H_n'(x)=nH_{n-1}(x)$. Therefore we get
\begin{align*}
&\frac{\partial}{\partial x_1}
\left(\sum_{m_1+\ldots+m_p=m+1}\frac{(m+1)!}{m_1!\cdots m_p!}\prod_{j=1}^pa_j^{m_j}H_{m_j}(x_j)\right)\\
=&\sum_{m_1+\ldots+m_p=m+1}\frac{(m+1)!}{(m_1-1)!\cdots m_p!}a_1^{m_1}H_{m_1-1}(x_1)\prod_{j=2}^pa_j^{m_j}H_{m_j}(x_j)\\
=&a_1(m+1)\sum_{m_1+\ldots+m_p=m}\frac{(m)!}{m_1!\cdots m_p!}\prod_{j=1}^pa_j^{m_j}H_{m_j}(x_j)\\
=&a_1(m+1)H_m\left(\sum_{j=1}^pa_jx_j\right)\\
=&\frac{\partial}{\partial x_1}H_{m+1}\left(\sum_{j=1}^pa_jx_j\right)
\end{align*}
The other derivatives can be handled similarly. Therefore \eqref{Runge} holds up to a constant. Let $x_1=\ldots=x_p=0$. If $m$ is odd both sides of \eqref{Runge} are equal to zero and thus the constant vanishes. For even $m$ we have $H_m(0)=(-1)^{m/2}(m-1)!!$, where
\begin{equation*}
(m-1)!!:=(m-1)(m-3)\cdots3\cdot1=\frac{m!}{2^{m/2}(m/2)!}.
\end{equation*}
This yields
\begin{align*}
&\sum_{m_1+\ldots+m_p=m}\frac{m!}{m_1!\cdots m_p!}\prod_{j=1}^pa_j^{m_j}H_{m_j}(0)\\
=&\sum_{2m_1+\ldots+2m_p=m}\frac{m!}{(2m_1)!\cdots (2m_p)!}\prod_{j=1}^p(-1)^{m_j}a_j^{2m_j}(2m_j-1)!!\\
=&(-1)^{m/2}\sum_{2m_1+\ldots+2m_p=m}\frac{m!}{(2m_1)!!\cdots (2m_p)!!}\prod_{j=1}^p(a_j^2)^{m_j}\\
=&(-1)^{m/2}\sum_{m_1+\ldots+m_p=m/2}\frac{m!}{2^{m/2}m_1!\cdots m_p!}\prod_{j=1}^p(a_j^2)^{m_j}\\
=&(-1)^{m/2}\sum_{m_1+\ldots+m_p=m/2}\frac{(m-1)!!2^{m/2}(m/2)!}{2^{m/2}m_1!\cdots m_p!}\prod_{j=1}^p(a_j^2)^{m_j}\\
=&(-1)^{m/2}(m-1)!!\sum_{m_1+\ldots+m_p=m/2}\frac{(m/2)!}{m_1!\cdots m_p!}\prod_{j=1}^p(a_j^2)^{m_j}\\
=&H_m(0)\left(\sum_{j=1}^p a_j^2\right)^{m/2}\\
=&H_m(0).
\end{align*}
\end{proof}
\begin{proof}[Proof of \thref{non central mv}]
By \thref{reduction mv} it is enough to study the limit of
\begin{equation*}
\left\{d_N^{-1}\sum_{j=1}^{\lfloor Nt\rfloor}\sum_{l_1+\ldots+l_p=m}
\frac{J_{l_1,\ldots,l_p}(x)}{l_1!\cdots l_p!}H_{l_1,\ldots,l_p}(X_j):(x,t)\in[-\infty,\infty]^p\times[0,1]\right\}.
\end{equation*}
We first show that in the current situation Lemma \ref{LemmaRunge} can be applied.
For all $k_1,\ldots,k_p$ satisfying $k_1+\ldots+k_p=m$ we can find real numbers $a_{k_1,\ldots,k_p}^{(1)},\ldots a_{k_1,\ldots,k_p}^{(p)}$, s.t. the matrix
\begin{equation*}
A=\left(\prod_{i=1}^p(a_{k_1,\ldots,k_p}^{(i)})^{m_i}\right)_{\substack{m_1+\ldots+m_p=m\\k_1+\ldots+k_p=m}}
\end{equation*}
is invertible. After normalization we have $\sum_{i=1}^p (a_{k_1,\ldots,k_p}^{(i)})^2=1$. For a suitable diagonalmatrix $M$ of the same size define $B:=MA^{-1}$, $B=(b(k_1,\ldots,k_p,l_1,\ldots,l_p))$, s.t.
\begin{align*}
&\sum_{k_1+\ldots+k_p=m}b(k_1,\ldots,k_p,l_1,\ldots,l_p)(a_{k_1,\ldots,k_p}^{(1)})^{m_1}\cdots
(a_{k_1,\ldots,k_p}^{(p)})^{m_p}\\
&=\begin{cases}(m!)^{-1}\prod_{i=1}^p l_i!~~~~&\textnormal{ if }(m_1,\ldots,m_p)=(l_1,\ldots,l_p)\\0
&\textnormal{ otherwise.}\end{cases}\numberthis\label{VorbereitungRunge}
\end{align*}
By using Lemma \ref{LemmaRunge} together with \eqref{VorbereitungRunge} we get
\begin{align*}
&\sum_{\substack{l_1+\ldots+l_p=m\\ k_1+\ldots+k_p=m}}J_{l_1,\ldots,l_p}(x)\left(\prod_{i=1}^p(l_i!)^{-1}\right)
b(k_1,\ldots,k_p,l_1,\ldots,l_p)H_m\left(\sum_{i=1}^p a_{k_1,\ldots,k_p}^{(i)}X_j^{(i)}\right)\\
=&\sum_{\substack{l_1+\ldots+l_p=m\\ k_1+\ldots+k_p=m}}\sum_{m_1+\ldots+m_p=m}
J_{l_1,\ldots,l_p}(x)\left(\prod_{i=1}^p(l_i!)^{-1}\right)b(k_1,\ldots,k_p,l_1,\ldots,l_p)\\
&~~~~~~~~~~~~~~~~\times m!\prod_{i=1}^p(m_i!)^{-1}\left(a_{k_1,\ldots,k_p}^{(i)}\right)^{m_i}H_{m_i}\left(X_j^{(i)}\right)\\
=&\sum_{l_1+\ldots+l_p=m}J_{l_1,\ldots,l_p}(x)\prod_{i=1}^p(l_i!)^{-1}H_{l_i}\left(X_j^{(i)}\right)
\end{align*}
For simplicity define
\begin{equation*}
I(x;k_1,\ldots,k_p):=\sum_{l_1+\ldots+l_p=m}J_{l_1,\ldots,l_p}(x)\left(\prod_{i=1}^p(l_i!)^{-1}\right)
b(k_1,\ldots,k_p,l_1,\ldots,l_p)
\end{equation*}
so that we get the identity
\begin{align*}
&d_N^{-1}\sum_{j=1}^{\lfloor Nt\rfloor}\sum_{l_1+\ldots+l_p=m}\prod_{i=1}^p(l_i!)^{-1}H_{l_i}\left(X_j^{(i)}\right)\\
=&d_N^{-1}\sum_{j=1}^{\lfloor Nt\rfloor}\sum_{k_1+\ldots+k_p=m}I(x;k_1,\ldots,k_p)H_m\left(\sum_{i=1}^p 
a_{k_1,\ldots,k_p}^{(i)}X_j^{(i)}\right)
\end{align*}
Note that $Y_j^{(k_1,\ldots,k_p)}:=\sum_{i=1}^p a_{k_1,\ldots,k_p}^{(i)}X_j^{(i)}$ is standard normal distributed and that 
\begin{align*}
&\int_{\RR^m}^{''}\frac{e^{it(x_1+\ldots+x_m)}-1}{i(x_1+\ldots+x_m)}\prod_{j=1}^m|x_j|^{-(1-D)/2}
\left(\sum_{i=1}^p a_{k_1,\ldots,k_p}^{(i)}B^{(i)}\right)(dx_1)\cdots 
\left(\sum_{i=1}^p a_{k_1,\ldots,k_p}^{(i)}B^{(i)}\right)(dx_m)\\
=&\sum_{j_1,\ldots,j_m=1}^p a_{k_1,\ldots,k_p}^{(j_1)}\cdots a_{k_1,\ldots,k_p}^{(j_m)}
\int_{\RR^m}^{''}\frac{e^{it(x_1+\ldots+x_m)}-1}{i(x_1+\ldots+x_m)}\prod_{j=1}^m|x_j|^{-(1-D)/2}
B^{(j_1)}(dx_1)\cdots B^{(j_m)}(dx_m)\\
:=&\sum_{j_1,\ldots,j_m=1}^p a_{k_1,\ldots,k_p}^{(j_1)}\cdots a_{k_1,\ldots,k_p}^{(j_m)}
Z_{j_1,\ldots,j_m}(t).
\end{align*}
Therefore, as in \eqref{Hermite gemeinsam} we have
\begin{gather*}
\left\{d_N^{-1}\sum_{j=1}^{\lfloor Nt\rfloor}H_m(Y_j^{(k_1,\ldots,k_p)}):k_1,\ldots,k_p=m, t\in[0,1]\right\}\xrightarrow{~~d~~}\\
\left\{\sum_{j_1,\ldots,j_m=1}^p a_{k_1,\ldots,k_p}^{(j_1)}\cdots a_{k_1,\ldots,k_p}^{(j_m)}
Z_{j_1,\ldots,j_m}(t):k_1,\ldots,k_p=m, t\in[0,1]\right\}.
\end{gather*}
By Dudley and Wichura's almost sure representation theorem we can find vector processes $(\tilde{S}_N(t))$ and $(\tilde{Z}(t))$, which have the same distribution as the above,  s.t. $(\tilde{S}_N(t))$ converges a.s. to $(\tilde{Z}(t))$ in $D[0,1]$. Since the functions $I(x;k_1,\ldots,k_p)$ are bounded, these a.s. convergence still holds in  $D([-\infty,\infty]^p\times [0,1])$ if one multiplies $I(x;k_1,\ldots,k_p)$ to the corresponding component of $(\tilde{S}_N(t))$ resp. $(\tilde{Z}(t))$ in $D[0,1]$. Applying the continious mapping theorem we get
\begin{multline*}
\left\{d_N^{-1}\sum_{j=1}^{\lfloor Nt\rfloor}\sum_{k_1+\ldots+k_p=m}I(x;k_1,\ldots,k_p)H_m(Y_j^{(k_1,\ldots,k_p)}):
(x,t)\in[-\infty,\infty]^p\times [0,1]\right\}\xrightarrow{~~d~~}\\
\Biggl\{\sum_{j_1,\ldots,j_m=1}^p\sum_{k_1+\ldots+k_p=m}I(x;k_1,\ldots,k_p)
 a_{k_1,\ldots,k_p}^{(j_1)}\cdots a_{k_1,\ldots,k_p}^{(j_m)}Z_{j_1,\ldots,j_m}(t):\\
 (x,t)\in[-\infty,\infty]^p\times [0,1]\Biggr\}.
\end{multline*}
Finally we have to verify that this limit is equal to \eqref{Grenzprozess}. By \eqref{VorbereitungRunge} we obtain
\begin{align*}
&\sum_{j_1,\ldots,j_m=1}^p\sum_{k_1+\ldots+k_p=m}
I(x;k_1,\ldots,k_p) a_{k_1,\ldots,k_p}^{(j_1)}\cdots a_{k_1,\ldots,k_p}^{(j_m)}\\
=&\sum_{j_1,\ldots,j_m=1}^p\sum_{\substack{k_1+\ldots+k_p=m\\l_1+\ldots+l_p=m}}
J_{l_1,\ldots,l_p}(x)\left(\prod_{i=1}^p(l_i!)^{-1}\right)
b(k_1,\ldots,k_p,l_1,\ldots,l_p)a_{k_1,\ldots,k_p}^{(j_1)}\cdots a_{k_1,\ldots,k_p}^{(j_m)}\\
=&\begin{cases}(m!)^{-1}J_{l_1,\ldots,l_p}(x)~~~~&\textnormal{ if }l_i=i(j_1,\ldots,j_m)\\0
&\textnormal{ otherwise.}\end{cases}
\end{align*}
But if $l_i=i(j_1,\ldots,j_m)$ for $1\leq i\leq p$ we have
\begin{equation*}
\tilde{J}_{j_1,\ldots,j_m}=(m!)^{-1}J_{l_1,\ldots,l_p}(x),
\end{equation*}
which completes the proof.
\end{proof}
\newpage
\bibliography{literatur}
\end{document}